\documentclass[a4paper,reqno,12pt]{amsart}

\usepackage{amsmath, amssymb, amsthm, enumerate,amsfonts,latexsym, mathrsfs}

\usepackage[top=30mm,right=27mm,bottom=30mm,left=27mm]{geometry}

\newtheorem{theorem}{Theorem}[section]

\newtheorem{lemma}[theorem]{Lemma}
\newtheorem{proposition}[theorem]{Proposition}

\makeatletter
\newcommand{\imod}[1]{\allowbreak\mkern4mu({\operator@font mod}\,\,#1)}
\makeatother

\newtheorem*{thmA}{Theorem~A}
\newtheorem*{thmB}{Theorem~B}
\newtheorem*{thmC}{Theorem~C}
\newtheorem*{thmD}{Theorem~D}
\newtheorem*{thmE}{Theorem~E}
\newtheorem*{thmF}{Theorem~F}
\newtheorem*{conj1}{Conjecture~1}

\newcommand{\Irr}{{\mathrm {Irr}}}

\newcommand{\Aut}{{\mathrm {Aut}}}

\newcommand{\PSL}{{\mathrm {PSL}}}

\newcommand{\SL}{{\mathrm {SL}}}

\newcommand{\PGL}{{\mathrm {PGL}}}

\newcommand{\SSS}{\mathrm{S}}
\newcommand{\AAA}{\mathrm{A}}

\newcommand{\Atlas}{{\sf Atlas}}
\newcommand{\acd}{\mathrm{acd}}

\newcommand{\nz}{\mathrm{nz}}
\newcommand{\anz}{\mathrm{anz}}

\theoremstyle{definition}

\newtheorem{rem}[theorem]{Remark}

\begin{document}
\title[Average number of zeros of characters]{Average number of zeros of characters of finite groups}

\author{Sesuai Yash Madanha}
\address{Department of Mathematics and Applied Mathematics, University of Pretoria, Pretoria, 0002, South Africa}
\email{sesuai.madanha@up.ac.za}

\thanks{}


\date{\today}


\begin{abstract}
There has been some interest on how the average character degree affects the structure of a finite group. We define, and denote by $ \anz(G) $, the average number of zeros of characters of a finite group $ G $ as the number of zeros in the character table of $ G $ divided by the number of irreducible characters of $ G $.  We show that if $ \anz(G) < 1 $, then the group $ G $ is solvable and also that if $ \anz(G) < \frac{1}{2} $, then $ G $ is supersolvable. We characterise abelian groups by showing that $ \anz(G) < \frac{1}{3} $ if and only if $ G $ is abelian. 
\end{abstract}

\maketitle
\begin{center} 
\textit{Dedicated to the memory of Kay Magaard}
\end{center}

\section{Introduction}\label{s:intro}

Let $ G $ be a finite group and $ \Irr (G) $ be the set of complex irreducible characters of $G$. Let $T(G)$ be the sum of degrees of complex irreducible characters of $G$, that is, $ T(G)=\sum _{\chi \in \Irr (G)}\chi (1) $. Denote by $k(G)$ the number of conjugacy classes of $G$. Then $ k(G)=|\Irr (G)| $. Define the average character degree of $ G $ by
\begin{center}
$ \acd(G):=\dfrac{T(G)}{|\Irr (G)|} $.
\end{center}
Recently, a lot of authors have investigated the average character degree of finite groups and how it influences the structure of the groups (see \cite{MT-V11,MN14,MN15}). In fact, Magaard and Tong-Viet \cite{MT-V11} proved that $ G $ is solvable whenever $ \acd(G) < 2 $ and they conjectured that their result still holds if $ \acd(G) \leqslant 3 $. This conjecture was settled by Isaacs, Loukaki and Moret\'o \cite[Theorem A]{ILM13} and they obtained some sufficient conditions for a group to be supersolvable and also to be nilpotent. In the same article, the authors of \cite{ILM13} conjectured that the best possible bound is when $ \acd(G) < \frac{16}{5} $ for $ G $ to be solvable. Moret\'o and Nguyen showed in \cite[Theorem A]{MN14}, that indeed this was the best bound. We shall state the results on average character degrees with the best bounds below. 

\begin{theorem} \cite[Theorem A]{MN14}\label{MN14TheoremA}
Let $ G $ be a finite group. If $ \acd(G) < \frac{16}{5} $, then $ G $ is solvable.
\end{theorem}

\begin{theorem} \cite[Theorem B]{ILM13}\label{ILM13TheoremB}
Let $ G $ be a finite group. If $ \acd(G) < \frac{3}{2} $, then $ G $ is supersolvable.
\end{theorem}

\begin{theorem} \cite[Theorem C]{ILM13}\label{ILM13TheoremC}
Let $ G $ be a finite group. If $ \acd(G) < \frac{4}{3} $, then $ G $ is nilpotent.
\end{theorem}
More work on average character degrees of finite groups is found in \cite{HT17, Hun17, HT20}. Another invariant that has been studied is the so-called average class size. We shall refer the reader to \cite{GR06, ILM13, Qia15} for more bibliographic information.

In this article, we investigate the corresponding problem for zeros of characters of finite groups. We have to define a new invariant first. Recall that if $ \chi (g)=0 $ for some $ g\in G $ and $ \chi \in \Irr (G) $, we say $ \chi $ vanishes on $ g $. So $ \chi $ vanishes on conjugacy classes. In the same spirit, we define $ \nz(G) $ to be the number of zeros in the character table of $ G $ and the \textit{average number of zeros of characters} of $ G $ by:
\begin{center}
$ \anz(G):= \dfrac{\nz(G)}{|\Irr (G)|} $.
\end{center}
Since linear characters do not vanish on any conjugacy class, $ \anz(G)=0 $ for an abelian group $ G $. A classical theorem of Burnside \cite[Theorem 3.15]{Isa06} shows that $ \chi (g)=0 $ for some $ g\in G $ and a non-linear irreducible character $ \chi $, that is, $ \chi $ vanishes on some conjugacy class. This means $ \anz(G) > 0 $ for non-abelian groups. We show that zeros of characters influence the structure of a finite group.

We begin by proving a result analogous to \cite[Theorem 1.1]{MT-V11}:

\begin{thmA}
If $ N $ is a minimal non-abelian normal subgroup of $ G $, then there exists a non-linear character $ \chi \in \Irr(G) $ such that $ \chi _{N} $ is irreducible and $ \chi $ vanishes on at least two conjugacy classes of $ G $.
\end{thmA}

We need Theorem A to prove our second result below which corresponds to Theorem \ref{MN14TheoremA}:

\begin{thmB}
Let $ G $ be a finite group. If $ \anz(G) < 1 $, then $ G $ is solvable.
\end{thmB}
Note that since $ \anz(\mathrm{A}_{5})=1 $, this bound is sharp.

We show that the converse of Theorem B does not necessarily hold. Let $ Q $ be a Sylow $ 3 $-subgroup of $ S= \PSL_{3}(7) $ and suppose $ G=N_{S}(Q) $. We have that $ \anz(G)= 7/6>1 $.

In Theorems \ref{ILM13TheoremB} and \ref{ILM13TheoremC}, the bounds are optimal since $ \acd(\mathrm{A}_{4})=\frac{3}{2} $ and $ \acd(\mathrm{S}_{3})=\frac{4}{3} $. Since $ \anz(\mathrm{A}_{4})=\frac{1}{2} $ and $ \anz(\mathrm{S}_{3})=\frac{1}{3} $, the following corresponding results hold:
\begin{thmC}
Let $ G $ be a finite group. If $ \anz(G) < \frac{1}{2} $, then $ G $ is supersolvable.
\end{thmC}

\begin{thmD}
Let $ G $ be a finite group. If $ \anz(G) < \frac{1}{3} $, then $ G $ is nilpotent.
\end{thmD}

Note that for an abelian group $ G $, $ \acd(G) = 1 $ and $ \anz(G)=0 $. We show in Remark \ref{Abeliangroups} that there is no  number $ c > 1 $ such that $ \acd(G) < c $ implies that $ G $ is abelian. Contrary to $ \acd(G) $, we show that there exists a bound $ c > 0 $, such that $ \anz(G) < c $ implies that $ G $ is abelian. Hence we obtain a non-trivial characterisation of abelian groups in terms of zeros of characters:

\begin{thmE}
Let $ G $ be a finite group. Then $ G $ is abelian if and only if $ \anz(G) < \frac{1}{3} $.
\end{thmE}
When $ G $ is of odd order we obtain a bound that is better than that in Theorem C. The following result is analogous to \cite[Theorem D(a)]{ILM13}.

\begin{thmF}
Let $ G $ be a finite group of odd order. If $ \anz(G) < 1 $, then $ G $ is supersolvable.
\end{thmF}
However, the bound in Theorem F might not be optimal. Indeed, in \cite[Theorem D(a)]{ILM13}, the bound is best possible since there exists a group $ G $ of odd order which is not supersolvable such that $ \acd(G)=\frac{27}{11} $ and also $ \anz(G)=\frac{16}{11} $ ($ G $ is the unique non-abelian group of order $ 75 $). Hence we propose the following conjecture:

\begin{conj1}
Let $ G $ be a finite group of odd order. If $ \anz(G) < \frac{16}{11} $, then $ G $ is supersolvable.
\end{conj1}

The article is organized as follows. In Section \ref{pre}, we list some preliminary results. In Section \ref{solvableandnonsolvable}, we prove Theorems A and B and in Section \ref{supersolvableuptoabelian} we prove Theorems C to F. We also show that there exist an infinite family of non-abelian nilpotent groups $ G $ such that $ \acd(G) < \frac{4}{3} $ in this last section.

\section{Preliminaries}\label{pre}

\begin{lemma}\label{extendible}
Suppose that $ N $ is a minimal normal non-abelian subgroup of a group $ G $. Then there exists an irreducible character $ \theta $ of $ N $ such that $ \theta $ is extendible to $ G $ with $ \theta (1)\geqslant 5 $. In particular, if $ N $ is simple such that
\begin{itemize}
\item[(a)] $ N $ is isomorphic to $ \AAA_{5} $, then $ \theta(1)=5 $
\item[(b)] $ N $ is isomorphic to $ \AAA_{n} $, $ n\geqslant 7 $, then $ \theta(1)=n-1 $
\item[(c)] $ N $ is isomorphic to a finite group of Lie type defined over a finite field distinct from the Tits group $ ^{2}\mathrm{F}_{4}(2)' $ and $ \PSL_{2}(5) $, then $ \theta $ is the Steinberg character.
\end{itemize}
\end{lemma}
\begin{proof}
The first assertion is the statement of \cite[Theorem 1.1]{MT-V11} and the second assertion follows from the proof of \cite[Theorem 1.1]{MT-V11}.
\end{proof}

\begin{lemma}\cite[Corollary 6.17]{Isa06}\label{injectivemap} Let $ N $ be a normal subgroup of $ G $ and let $ \chi \in \Irr(G) $ be such that $ \chi _{N}=\theta \in \Irr(N) $. Then the characters $ \beta \chi $ for $ \beta \in \Irr (G/N) $ are irreducible, distinct for distinct $ \beta $ and are all of the irreducible constituents of $ \theta ^{G} $.
\end{lemma}
Recall that $ \Irr(G{\mid} K)= \{\chi \in \Irr(G): K\nsubseteq  \ker  \chi\} $. For $ \chi \in \Irr (G) $ we write  $ n\upsilon(\chi )$ for the number of conjugacy classes of  $G$ on which $\chi$  vanishes. Then $ \nz(G{\mid} K):=\sum _{\chi \in \Irr(G\mid K)}n\upsilon(\chi)  $.

\begin{lemma}\label{noofzerosfactorgroup}
Let $ K $ be a normal subgroup of $ G $ such that $ K\leqslant G' $. If $ \anz(G) < 1 $, then $ \anz(G/K)\leqslant \anz(G) $.
\end{lemma}
\begin{proof}
Let $ a=\anz(G/K) $. Note that $ k(G)=|\Irr (G)|=|\Irr(G/K)|+|\Irr(G{\mid} K)| $. If $ |\Irr(G/K)|=c $, $ |\Irr(G{\mid} K)|=d $, $ \nz(G/K)=m $ and $ \nz(G{\mid} K)=n $, then $ a=\frac{m}{c} $. Since $ K\leqslant G' $, $ \Irr(G{\mid} K) $ is a set of non-linear irreducible characters of $ G $ and also using Burnside's Theorem, $ |\Irr(G{\mid} K)|= d\leqslant n=\nz(G{\mid} K) $. Since the character table of $ G/K $ is a sub-table of the character of $ G $, we have that $ \nz(G/K)\leqslant \nz(G) $. Hence $ 1 > \anz(G)=\frac{m + n}{c + d}\geqslant \frac{m + d}{c + d}\geqslant \frac{m}{c}=a $ as required.
\end{proof}

We show that Theorem B holds for perfect groups.

\begin{lemma}\label{G=G'holds}
Let $ G $ be a finite group such that $ G=G' $. Then $ \anz(G)\geqslant 1 $. 
\end{lemma}
\begin{proof}
Since $ G $ has one linear character, it is sufficient to show that there exists a non-linear irreducible character of $ G $ that vanishes on at least two conjugacy classes. Suppose the contrary, that is, every non-linear irreducible character of $ G $ vanishes on exactly one conjugacy class. Then by \cite[Proposition 2.7]{Chi99}, $ G $ is a Frobenius group with
a complement of order $2$ and an abelian odd-order kernel, that is, $ G $ is solvable, contradicting the hypothesis that $ G $ is perfect. Hence the result follows.
\end{proof}

Let $ G $ be a finite group and $ \chi \in \Irr(G) $. Recall that $\upsilon(\chi ):= \{x\in G\mid \chi (x)=0\} $.

\begin{lemma}\label{ZSS10Lemma2.1} Let $ G $ be a non-abelian finite group and $ \chi \in \Irr(G) $ be non-linear. Suppose that $ N $ is a normal subgroup of $ G $ such that $ G'\leqslant N < G $. If $ \chi_{N} $ is not irreducible, then the following two statements hold:
\begin{itemize}
\item[(a)] There exists a normal subgroup $ K $ of $ G $ such that $ N\leqslant K < G $ and $ G{\setminus} K \subseteq \upsilon(\chi) $.
\item[(b)] If $ (G{\setminus} G')\cap \upsilon(\chi) $ consists of $ n $ conjugacy classes of $ G $, then 
\begin{center}
$ |G{:}G'|-|K{:}G'|\leqslant n $.
\end{center}
\end{itemize}
\end{lemma}
\begin{proof}
For (a), note that since $ G/N $ is abelian, it follows that $ \chi $ is a relative $ M $-character with respect to $ N $ by \cite[Theorem 6.22]{Isa06}. This means that there exists $ K $ with $ N\leqslant K \leqslant G $ and $ \psi \in \Irr(K) $ such that $ \chi =\psi ^{G} $ and $ \psi _{N}\in \Irr(N) $. Hence $ G{\setminus} K \subseteq \upsilon(\chi) $, Since $ \chi _{N} $ is not irreducible we have that $ K < G $ and the result follows.

For (b), let $ g_{1}, g_{2}, \dots, g_{m} $ be a complete set of representatives of the cosets of $ G' $ in $ G $, with $ g_{1}, g_{2}, \dots ,g_{k} $ a complete set of representatives of the cosets of $ G' $ in $ K $ so that $ m=|G:G'| $ and $ k=|K:G'| $. Since $ \chi $ vanishes on $ G{\setminus} K $, we have that $ \chi $ vanishes on $ g_{k+1}, g_{k+2}, \dots ,g_{m} $. But these elements are not conjugate and so $ m-k\leq n $ as required.
\end{proof}

In the following results we will consider groups with an irreducible character that vanishes on one conjugacy class. We denote by $ k_{G}(N) $, the number of conjugacy classes of $ G $ in $ N $ where $ N $ is a subset of $ G $. We first prove an easy lemma:

\begin{lemma}\label{Lemma2(2)}
Let $ N $ be a normal subgroup of $ G $. If $ k_{G}(G{\setminus} N)=1 $, then $ |G{:}N|=2 $ and $ G $ is a Frobenius group with an abelian kernel $ N $ of odd order.
\end{lemma}
\begin{proof}
Since $ k_{G}(G{\setminus} N)=1 $, all the elements in $ G{\setminus} N $ are conjugate. It follows that if $ x\in G{\setminus} N $, then the size of the conjugacy class containing $ x $ , say $ c $, is $ |G|-|N| $. Since $ N $ is a proper subgroup of $ G $, $ c\geq |G|-|G|/2=|G|/2 $.  But the size of any conjugacy class of a non-trivial group is a proper divisor of the group, so we deduce that $ |N|=|G|/2 $. Then $ C_{G}(x)=\langle x\rangle $ is of order $ 2 $, so $ G=\langle x \rangle N $ is a Frobenius group with a complement $ \langle x \rangle $ of order $ 2 $ and hence the Frobenius kernel $ N $ is abelian of odd order.
\end{proof}

We also need this result:

\begin{proposition}\label{classificationoneclass}
Let $ G $ be an almost simple group. If $ \chi \in \Irr (G) $ vanishes on exactly one conjugacy class, then one of the following holds:
\begin{itemize}
\item[(a)] $ G= \PSL_{2}(5) $, $ \chi (1)=3 $ or $ \chi (1)=4 $;
\item[(b)] $ G\in \{\AAA_{6}{:}2_{2},~ \AAA_{6}{:}2_{3}\} $, $ \chi(1)=9 $ for all such $ \chi\in \Irr (G) $;
\item[(c)] $ G=\PSL_{2}(7) $, $ \chi (1)= 3 $;
\item[(d)] $ G=\PSL_{2}(8){:}3 $, $ \chi (1)=7 $;
\item[(e)] $ G=\PGL_{2}(q) $, $ \chi (1)=q $, where $ q\geq 5 $;
\item[(f)] $ G=$  $^{2}\rm{B_{2}}(8){:}3 $, $ \chi(1)=14 $.
\end{itemize}
\end{proposition}

\begin{proof}
It was shown in \cite[Theorem 5.2]{Mad19q} and the proof of \cite[Theorem 1.2]{Mad19zp} that all irreducible characters of $ G $ that vanishes on exactly one conjugacy class are primitive, the result follows. Hence the list in \cite[Theorem 1.2]{Mad19zp} is a complete one. The result then follows.
\end{proof}

\begin{lemma}\label{Sei68} A finite group $ G $ has exactly one non-linear irreducible character if and only
if $ G $ is isomorphic to one of the following:
\begin{itemize}
\item[(a)] $ G $ is an extra-special $2$-group.
\item[(b)] $G$ is a Frobenius group with an elementary abelian kernel of order $ p^{n} $ and a cyclic complement of order $ p^{n}-1 $, where $ p $ is prime and $ n $ a positive integer.
\end{itemize}
Moreover, $ \anz(G)=\frac{m - 1}{ m + 1} $ for some integer $ m\geqslant 2 $.
\end{lemma}
\begin{proof}
The statements (a) and (b) follows from \cite[Theorem]{Sei68}. Let $ G $ be an extra-special $ 2 $-group of order $ 2^{2k+1} $ with $ |G/G'|=|G/Z(G)|=2^{2k} $. Then $ G $ has $ 2^{2k} $ linear characters and so $ 2^{2k} + 1 $ irreducible characters. Note that the non-linear irreducible character $ \chi $ of $ G $ is fully ramified with respect to $ Z(G) $. In particular, $ \chi $ vanishes on $ G{\setminus} Z(G) $. Hence $ \chi $ vanishes on at least $ 2^{2k} - 1 $ conjugacy classes. Since $ G $ has the identity and a non-trivial central element which are non-vanishing elements, $ \chi $ vanishes on exactly $ 2^{2k} - 1 $ conjugacy classes. Then $ \anz(G) = \frac{2^{2k} - 1}{2^{2k} + 1} $ and the result follows.

Suppose $ G $ is a Frobenius group with an elementary abelian kernel $ G' $ of order $ p^{n} $ and a cyclic complement of order $ p^{n} - 1 $. Note that $ |G/G'|=p^{n} - 1 $ and so $ G $ has $ p^{n} - 1 $ linear characters. Hence $ G $ has $ p^{n} $ irreducible characters. Since the non-linear irreducible character $ \chi $ of $ G $ is induced from irreducible character of $ G' $, $ \chi $ vanishes on $ G{\setminus}G' $.  It follows that $ \chi $ vanishes on at least $ p^{n} - 2 $ conjugacy classes. Since $ G' $ has one non-trivial conjugacy class of $ G $, $ \chi $ does not vanish on $ G' $. Otherwise, $ \chi $ vanishes on $ G{\setminus} \{1\}$, a contradiction. Since $ G $ has $ p^{n} $ conjugacy classes, we have that $ \chi $ vanishes on exactly $ p^{n} - 2 $ conjugacy classes and so $ \anz(G)=\frac{p^{n} - 2}{p^{n}} $ as required.
\end{proof}

We shall need the following two results to prove Theorem F. 

\begin{lemma}\cite{Pal81}\label{Pal81} A finite group $ G $ has exactly two non-linear irreducible characters if and only
if $ G $ is isomorphic to one of the following:
\begin{itemize}
\item[(a)] $ G $ is an extra-special $3$-group.
\item[(b)] $ G $ is a $ 2 $-group of order $2^{2k+2} $, $|G'| = 2 $, $ |Z(G)| = 4 $, and $ G $ has two non-linear
irreducible characters with equal degree $ 2^{k} $.
\item[(c)] $ G $ is a Frobenius group with an elementary abelian kernel of order $ 9 $ and Frobenius complement $ Q_{8} $.
\item[(d)] $ G $ is a Frobenius group with an elementary abelian kernel of order $ p^{k} $ and a cyclic Frobenius complement of order $ (p^{k} - 1)/2 $, where $ p $ is odd prime.
\item[(e)] $ G/Z(G) $ is a Frobenius group with an elementary abelian kernel of order $ p^{k} $ and a cyclic Frobenius complement of order $ p^{k} - 1 $, where $ p $ is prime and $ |Z(G)|=2 $.
\end{itemize}
\end{lemma}

\begin{lemma}\cite[Theorem 2.6]{Qia07}\label{Qia07Theorem2.6} Let $ G $ be a finite group of odd order. Then $ G $ has an irreducible
character that vanishes on exactly two conjugacy classes if and only if $ G $ is one of the following groups:
\begin{itemize}
\item[(a)] $ G $ is a Frobenius group with a complement of order $ 3 $.
\item[(b)] There are normal subgroups $ M $ and $ N $ of $ G $ such that: $ M $ is a Frobenius group with the
kernel $ N $; $ G/N $ is a Frobenius group of order $ p(p - 1)/2 $ with the kernel $ M/N $ and
a cyclic complement of order $ (p - 1)/2 $ for some odd prime $ p $. In this case, $ \chi_{M} $ is irreducible.

\end{itemize}

\end{lemma}

\begin{lemma}\label{numbertheoryresult}
Let $ \ell, m, n, b $ be positive integers with $ \ell,n \geqslant 2 $ and $ b \geqslant 2\ell + 1 $. Suppose that $ b= mn $. Then 
\begin{center}
$ b - m \geqslant \ell + 1 $.
\end{center}
\end{lemma}
\begin{proof}
Suppose $ b $ is odd. If $ b $ is prime, then $ b - 1\geqslant 2\ell + 1 - 1 > \ell + 1 $. If $ b $ is not prime, then the greatest $ m $ is when $ n=3 $ and $ b - m \geqslant b - \frac{b}{3}\geqslant \frac{2}{3}b > \frac{b}{2}=\ell $. If $ b $ is even, then $ b \geqslant 2\ell + 2 $ and the greatest $ m $ is when $ n=2 $. Hence $ b - m\geqslant 2\ell + 2 -(\ell + 1)=\ell + 1 $ as required.
\end{proof}

\section{Non-solvable and solvable groups}\label{solvableandnonsolvable}

\begin{proof}[{\textbf{Proof of Theorem A}}]
We begin the proof by showing that we may assume that $ N $ is simple. For if $ N=T_{1}\times T_{2}\cdots \times T_{k} $, where $ T_{i}\cong T $, $ T $ is a non-abelian simple group, for $ i=1,2, \dots ,k $ and $ k\geqslant 2 $, then by Lemma \ref{extendible}, there exists $ \theta_{i} \in \Irr(T_{i}) $ such that $ \chi _{N}=\theta_{1}\times \theta _{2}\times \cdots \times \theta_{k} $, where $ \chi \in \Irr(G) $. Since $ \theta_{1}(1)> 1 $, there exists $ x\in T_{1} $ such that $ \theta_{1}(x)=0 $ by Burnside's Theorem. But $ \chi $ vanishes on $ (x,1,\dots , 1) $ and $ (x,x,1,\dots, 1) $. Since these two elements are not conjugate in $ G $, $ \chi $ vanishes on at least two conjugacy classes of $ G $. 

We may assume that $ N $ is simple. Let $ C=C_{G}(N) $. We claim that $ C=1 $. Otherwise, $ N\times C $ is a normal subgroup of $ G $. There exists a non-linear $ \theta \in \Irr(N) $ such that $ \theta _{N}=\chi $ by Lemma \ref{extendible}. There exists $ x\in N $ such that $ \theta (x)=0 $. Then $ \chi(xc)=0 $ for any $ c\in C $. Let $ c\in C{\setminus}\{1\} $. Since $ x $ and $ xc $ are not conjugate in $ G $, $ \chi $ vanishes on at least two conjugacy classes. Hence $ C=1 $ and we have that $ G $ is almost simple.
 
By Lemma \ref{extendible}, there exists $ \chi $ such that $ \chi_{N} $ is irreducible and $ \chi(1)\geqslant 5 $. If $ \chi $ vanishes on two conjugacy classes, then the result follows. We may assume that $ \chi $ vanishes on one conjugacy class. By Proposition \ref{classificationoneclass} and considering the choice of $ \chi $  from Lemma \ref{extendible}, we are left with the following cases: $ G\in \{ \PGL_{2}(q), \AAA_{6}{:}2_{2}, \AAA_{6}{:}2_{3}\} $(note that $ \PSL_{2}(9)\cong \AAA_{6} $). Note that $ N=\PSL_{2}(q) $. We will choose another irreducible character of $ N $ that extends to $ G $. Obviously that alternative character vanishes on at least two conjugacy classes. For $ G=\SSS_{5} $, let $ \theta\in \Irr(\AAA_{5}) $ be such that $ \theta(1)=4 $. Then $ \theta $ is extendible to $ \Aut(\AAA_{5})=\SSS_{5} $. Using the \Atlas {} \cite{CCNPW85}, we have that $ \chi_{G} $ vanishes on two elements of distinct orders. Hence the result follows. 

For $ G\in \{ \AAA_{6}{:}2_{2}, \AAA_{6}{:}2_{3} \}$, using the \Atlas{} \cite{CCNPW85}, we can choose an irreducible character $ \theta $ of $ N $ of degree $ 10 $ that extends to $ \Aut(\AAA_{6}) $. Lastly, $ \chi $ vanishes on more than two conjugacy classes of $ G $. Hence the result follows.

Suppose $ G = \PGL_{2}(q) $, where $ q\geqslant 7 $. It is well known that $ \PSL_{2}(q) $ has irreducible characters of degree $ q-1 $ and $ q + 1 $. By \cite[Theorem A]{Whi13}, an irreducible character $ \theta $ of degree $ q + 1 $ is extendible to $ \Aut(\PSL_{2}(q)) $ except when $ N=\PSL_{2}(3^{f}) $ and $ G=\PGL_{2}(3^{f}) $, with $ f $ an odd positive integer (case (iii) of \cite[Theorem A]{Whi13}). If $ N=\PSL_{2}(3^{f}) $ and $ G=\PGL_{2}(3^{f}) $ with $ f $ an odd positive integer, then we choose an irreducible character $ \theta $ of $ \PSL_{2}(3^{f}) $ of degree $ q - 1 $ which is extendible to $ \Aut(\PSL_{2}(3^{f})) $. Thus the result follows.
\end{proof}

\begin{proof}[{\textbf{Proof of Theorem B}}] We shall prove our result by induction on $ |G| $. We may assume that $ G $ is non-solvable. If $ G=G' $, then the result follows by Lemma \ref{G=G'holds}, so $ G\neq G' $. 

Let $ G^{\infty} $ be the solvable residual of $ G $. Note that $ G^{\infty} $ is perfect. If $ N < G^{\infty} \leqslant G' $ is a minimal normal subgroup of $ G $, then $ G^{\infty}/N $ is perfect and so $ G/N $ is non-solvable. But $ \anz(G/N) < 1 $ by Lemma \ref{noofzerosfactorgroup} and hence $ G/N $ is solvable by induction, a contradiction. 

We may assume that $ N=G^{\infty} $ is a non-abelian minimal normal subgroup of $ G $. By Theorem A, there exists $ \chi \in \Irr(G) $ such that $ \chi_{N} $ is irreducible and $ \chi $ vanishes on two conjugacy classes. Suppose the two conjugacy classes are represented by elements $ g_{1} $ and $ g_{2} $ of $ G $. Since $ \chi_{G'} $ is irreducible, we have that the character $ \beta\chi  $ is irreducible for every linear $ \beta $ of $ G $ by Lemma \ref{injectivemap}. The $ \beta\chi $'s are distinct for distinct characters $ \beta $. We show that every character of the form $ \beta\chi  $ also vanishes on $ g_{1} $ and $ g_{2} $. Then $ \beta\chi(g_{i})=\beta(g_{i})\chi({g_{i}})=\beta(g_{i})\cdot 0=0 $, where $ i\in \{ 1, 2\} $.

Hence for every linear character $ \beta $ of $ G $, there is a corresponding non-linear irreducible character of  $ G $ of the form $ \beta\chi  $ that vanishes on two conjugacy classes of $ G $. Let $ a $ be the number of linear characters of $ G $ and $ b $ be the number of non-linear irreducible characters not of the form  $ \beta\chi  $ (note that these irreducible characters may vanish on more than one conjugacy class). Then $ |\Irr(G)|=2a + b $ and $ \nz(G)=2a + b + c $, where $ c $ is a non-negative integer. Therefore $ \anz(G)=\frac{2a + b + c}{2a + b}\geqslant \frac{2a + b}{2a + b}\geqslant \frac{2a}{2a}= 1 $, concluding our argument.
\end{proof}

\section{Supersolvable, nilpotent and abelian groups}\label{supersolvableuptoabelian}

\begin{proof}[{\textbf{Proof of Theorem C}}]
Suppose $ G $ is non-abelian, that is, $ G' > 1 $. We shall use induction on $ |G| $ to show that $ G $ is supersolvable. By Theorem B, $ G $ is solvable since $ \anz(G) < \frac{1}{2} < 1 $. Let $ N\leqslant G' $ be a minimal normal subgroup of $ G $. By Lemma \ref{noofzerosfactorgroup}, $ \anz(G/N)\leqslant \anz(G) < \frac{1}{2} $. Using the inductive hypothesis, $ G/N $ is supersolvable. If $ N $ is cyclic or $ N\leqslant \Phi(G) $, the Frattini subgroup, then $ G $ is supersolvable and the result follows. 

Suppose that $ \chi \in \Irr(G) $ be non-linear such that $ \chi_{G'} $ is irreducible. By Lemma \ref{injectivemap}, for every linear character $ \beta_{i} \in \Irr(G/G') $, there exists a non-linear $ \beta_{i}\chi  \in \Irr(G) $, that is, the number of linear characters of $ G $ is less than or equal to the number of non-linear irreducible characters of $ G $. Every non-linear irreducible character of $ G $ vanishes on at least one conjugacy class by \cite[Theorem 3.15]{Isa06}. Let $ |\Irr(G)|=a + b $, where $ a $ is the number of the non-linear characters and $ b $ is the number of linear characters and $ \nz(G)=a + c $, where $ c $ is non-negative integer. Then $ \anz(G)=\frac{a + c}{a + b}\geqslant \frac{a + c}{2a} \geqslant\frac{1}{2} $, since $ a\geqslant b $. This contradicts our hypothesis.

Hence every non-linear irreducible character $ \chi $ of $ G $ is such that $ \chi_{G'} $ is not irreducible. Note that $ G $ is an $ M $-group by \cite[Theorems 6.22 and 6.23 ]{Isa06} since $ G/N $ is supersolvable and $ N $ is abelian. By Lemma \ref{ZSS10Lemma2.1}, there exists a normal subgroup $ K $ of $ G $ such that $ G' \leqslant K < G $ and $ G{\setminus} K \subseteq \upsilon(\chi) $. If $ \upsilon(\chi) $ is a conjugacy class, then $ G $ is a Frobenius group with an abelian kernel and a complement of order two by Lemma \ref{Lemma2(2)}. Thus by \cite[Proposition 2.7]{Chi99}, every irreducible character of $ G $ vanishes on at most one conjugacy class. Note that $ G $ has two linear characters. Since $ \anz(G) < \frac{1}{2} $, $ G $ can only have one non-linear character. It follows that $ G $ has three conjugacy classes. 
This implies that $ G\cong \SSS_{3} $. Now $ \SSS_{3} $ is supersolvable group and we are done.

Therefore every non-linear irreducible character of $ G $ vanishes on at least $ \ell $ conjugacy classes, where $ \ell \geqslant 2 $. Suppose $ G $ has an irreducible character $ \chi $ that vanishes on exactly $ \ell $ conjugacy classes. Let $ b $ be the number of linear characters of $ G $ and $ m $ be the number of non-linear irreducible characters of $ G $. If $ b \leqslant 2\ell -1 $, then $ \anz(G)\geqslant \frac{m\ell}{b + m}\geqslant \frac{1}{2} $, contradicting our hypothesis. We may assume that $ b = 2\ell $. Then $ G $ has only one non-linear irreducible character $ \chi $. By Lemma \ref{Sei68}, $ G $ is a Frobenius group with an elementary abelian kernel of order $ p^{n} $ and cyclic complement of order $ p^{n}-1 $ for some prime $ p $. Also note that $ |\Irr(G)|=2\ell + 1 $. By Lemma \ref{Sei68}, $ \anz(G)=\frac{2\ell - 1}{2\ell + 1} \geqslant \frac{1}{2} $ since $ \ell \geqslant 2 $, a contradiction. 

Suppose that $ b \geqslant 2\ell + 1 $. Then by Lemma \ref{ZSS10Lemma2.1}, there exists a normal subgroup $ K $ of $ G $ such that $ G' \leqslant K < G $ and $ G{\setminus} K \subseteq \upsilon(\varphi) $ for every non-linear irreducible character $ \varphi $ of $ G $. We also have that $ |G{:}G'| - |K{:}G'| = b - | K{:}G'| \geqslant \ell + 1 $ by Lemma \ref{numbertheoryresult}, again contradicting our hypothesis that $ G $ has an irreducible character that vanishes on exactly $ \ell $ conjugacy classes. This concludes our proof.
\end{proof}

\begin{proof}[{\textbf{Proof of Theorem D}}]
Suppose $ \anz(G)<\frac{1}{3} $. By Theorem C, $ G $ is supersolvable group. In particular, $ G $ is an $ M $-group. Suppose that $ \chi \in \Irr(G) $ be non-linear such that $ \chi_{G'} $ is irreducible. Then by the argument in the second paragraph of the proof of Theorem C, $ \anz(G)\geqslant\frac{1}{2} > \frac{1}{3} $, a contradiction. Hence every non-linear irreducible character $ \chi $ of $ G $ is such that $ \chi_{G'} $ is not irreducible. There exists a normal subgroup $ K $ of $ G $ such that $ G' \leqslant K < G $ and $ G{\setminus} K \subseteq \upsilon(\chi) $ by Lemma \ref{ZSS10Lemma2.1}. If $ \upsilon(\chi) $ is a conjugacy class, then $ G $ is a Frobenius group with an abelian kernel and a complement of order two using Lemma \ref{Lemma2(2)}. By \cite[Proposition 2.7]{Chi99}, we have that every irreducible character of $ G $ vanishes on at most one conjugacy class. Since $ G $ has two linear characters, $ \anz(G) \geqslant \frac{1}{3} $, contradicting our hypothesis. 

Hence every non-linear irreducible character of $ G $ vanishes on at least $ \ell $ conjugacy classes, where $ \ell \geqslant 2 $. Suppose $ G $ has an irreducible character $ \chi $ that vanishes on exactly $ \ell $ conjugacy classes. Let $ b $ be the number of linear characters of $ G $. If $ b \leqslant 3\ell -1 $, then $ \anz(G)\geqslant \frac{1}{3} $, a contradiction. We may assume that $ b = 3\ell $. Then $ G $ has only one non-linear irreducible character $ \chi $. By Lemma \ref{Sei68}, $ G $ is a Frobenius group with an elementary abelian kernel of order $ p^{n} $ and a cyclic complement of order $ p^{n}-1 $ for some prime $ p $. By Lemma \ref{Sei68}, $ \anz(G)=\frac{2\ell - 1}{2\ell + 1} > \frac{1}{3} $ since $ \ell \geqslant 2 $, a contradiction. 

Suppose that $ b \geqslant 3\ell + 1 $. Then by Lemma \ref{ZSS10Lemma2.1}, there exists a normal subgroup $ K $ of $ G $ such that $ G' \leqslant K < G $ and $ G{\setminus} K \subseteq \upsilon(\varphi) $ for every non-linear irreducible character $ \varphi $ of $ G $. We also have that $ |G{:}G'| - |K{:}G'| = b - | K{:}G'| \geqslant \ell + 1 $ by Lemma \ref{numbertheoryresult}, contradicting our hypothesis that $ G $ has an irreducible character that vanishes on exactly $ \ell $ conjugacy classes. This concludes our proof. 
\end{proof}

\begin{rem}\label{Abeliangroups}
Let $ \mathcal{L} =\{ G \mid G $ is an extra-special $ 2 $-group of order $ 2^{2k + 1} $ for some positive integer  $ k \} $ and $ G\in \mathcal{L} $. Since $ |G/G'|=2^{2k} $ and $ G $ has only one non-linear irreducible character, we have $ \acd(G)=\frac{2^{2k} + 2^{k}}{2^{2k} + 1} $. Note that $ \mathcal{L} $ is an infinite family of nilpotent groups $ G $ such that $ \acd(G) < \frac{4}{3} $. Hence $ \acd(G) \rightarrow 1  $ and $ G\rightarrow \infty $. In other words there does not exist $ c > 1 $ such that $ \acd(G) < c $ implies that $ G $ is abelian.
\end{rem}

Our last result shows that there does not exist a non-abelian nilpotent group $ G $ such that $ \anz(G) < \frac{1}{3} $. We shall restate Theorem E below.

\begin{theorem}
Let $ G $ be a finite group. Then $ G $ is abelian if and only if $ \anz(G) < \frac{1}{3} $.
\end{theorem}

\begin{proof} If $ G $ is abelian, then $ \anz(G)=0 < \frac{1}{3} $. Suppose that $ \anz(G) < \frac{1}{3} $. Then $ G $ is nilpotent group by Theorem D. Using induction on $ |G| $, we have that $ G/N $ is abelian for some minimal normal subgroup $ N $ of $ G $, that is, $ N=G' $. If $ \chi_{G'} $ is irreducible, then $ \chi $ is linear. Hence every non-linear irreducible character $ \chi $ is such that $ \chi _{G'} $ is not irreducible. 

By Lemma \ref{ZSS10Lemma2.1}, there exists a normal subgroup $ K $ of $ G $ such that $ G' \leqslant K < G $ and $ G{\setminus} K \subseteq \upsilon(\chi) $. If $ \upsilon(\chi) $ is a conjugacy class, then $ G $ is a Frobenius group with an abelian kernel and a complement of order two by Lemma \ref{Lemma2(2)}, a contradiction since $ G $ is nilpotent.

Hence every non-linear irreducible character of $ G $ vanishes on at least $ \ell $ conjugacy classes, where $ \ell \geqslant 2 $. Suppose $ G $ has an irreducible character $ \chi $ that vanishes on $ \ell $ conjugacy classes. Let $ b $ be the number of linear characters of $ G $. If $ b \leqslant 3\ell -1 $, then $ \anz(G)\geqslant \frac{1}{3} $, a contradiction. We may assume that $ b = 3\ell $. Then $ G $ has only one non-linear irreducible character $ \chi $. By Lemma \ref{Sei68}, $ G $ is an extra-special group $ 2 $-group. Using the second part of Lemma \ref{Sei68}, we have that $ \anz(G)=\frac{2\ell - 1}{2\ell + 1} > \frac{1}{3} $ since $ \ell \geqslant 2 $, a contradiction. 

Suppose that $ b \geqslant 3\ell + 1 $. Then by Lemma \ref{ZSS10Lemma2.1}, there exists a normal subgroup $ K $ of $ G $ such that $ G' \leqslant K < G $ and $ G\setminus K \subseteq \upsilon(\varphi) $ for every non-linear irreducible character $ \varphi $ of $ G $. We have that $ |G{:}G'| - |K{:}G'| = b - | K{:}G'| \geqslant \ell + 1 $ by Lemma \ref{numbertheoryresult}, contradicting our hypothesis that $ G $ has an irreducible character that vanishes on exactly $ \ell $ conjugacy classes. This concludes our proof. 
\end{proof}

\begin{proof}[\textbf{Proof of Theorem F}] Suppose that $ G $ is non-abelian. Note that $ G $ is solvable. Also note that every non-linear irreducible character of $ G $ vanishes on at least $ \ell $ conjugacy classes, $ \ell \geqslant 2 $. Suppose that $ G $ has an irreducible character $ \chi $ that vanishes on two conjugacy classes. By Lemma \ref{Qia07Theorem2.6}, either $ G $ is a Frobenius group with a complement of order $ 3 $ or there are normal subgroups $ M $ and $ N $ of $ G $ such that: $ M $ is a Frobenius group with the kernel $ N $; $ G/N $ is a Frobenius group of order $ p(p - 1)/2 $ with the kernel $ M/N $ and a cyclic complement of order $ (p - 1)/2 $ for some odd prime $ p $. If $ G $ is a Frobenius group with a complement of order $ 3 $, then $ G $ can only have at most two non-linear irreducible characters. Otherwise $ \anz(G)\geqslant \frac{8}{7} $, a contradiction (note that a group of odd order has an even number of non-linear irreducible characters). By Lemma \ref{Pal81}, $ G $ is the group in the case (d) of Lemma \ref{Pal81} with $ (p^{k} - 1)/2=3 $, that is, $ p^{k} = 7 $ and $ |G|=21 $. Hence $ G $ is supersolvable. We may assume that $ G $ is the group in the second case above. Then $ M=G' $, $ \chi_{M} $ is irreducible and by Lemma \ref{injectivemap}, for every linear character $ \alpha $ of $ G $, there exists a corresponding $ \alpha \chi \in \Irr(G) $ and hence $ \anz(G) \geqslant 1 $, contradicting our hypothesis. 

Hence every non-linear irreducible character of $ G $ vanishes on at least $ \ell $ conjugacy classes, $ \ell \geqslant 3 $. Suppose that $ G $ has irreducible character vanishing on exactly $ \ell $ conjugacy classes. If there exists $ \chi $ such that $ \chi_{G'} $ is irreducible, then for every linear character $ \alpha $ of $ G $, there exists a corresponding $ \alpha \chi \in \Irr(G) $ and hence $ \anz(G) \geqslant 1 $, contradicting our hypothesis. We may assume that for every irreducible character $ \chi $ of $ G $, $ \chi_{G'} $ is reducible. By Lemma \ref{ZSS10Lemma2.1}, there exists a normal subgroup $ K $ of $ G $ such that $ G' \leqslant K < G $ and $ G{\setminus} K \subseteq \upsilon(\varphi) $ for every non-linear irreducible character $ \varphi $ of $ G $. If $ b \leqslant 2\ell - 1 $, then $ G $ has at most two non-linear irreducible characters. Otherwise, $ \anz(G)\geqslant \frac{3\ell}{2\ell + 2} > 1 $, a contradiction. Hence $ G $ has exactly two non-linear irreducible characters. By Lemma \ref{Pal81}, $ G $ is a Frobenius group with an elementary abelian kernel $ G' $ of order $ p^{k} $ and a cyclic Frobenius complement of order $ (p^{k} - 1)/2 $, where $ p $ is odd prime. Note that every non-linear irreducible character of $ G $ vanishes on $ G{\setminus} G' $. Thus $ b=(p^{k} - 1)/2 $ and since $ |G{:}G'|=b $, we have that $ b - 1 \leqslant \ell $ using Lemma \ref{ZSS10Lemma2.1}. Hence $ \anz(G)\geqslant \frac{2(b-1)}{b + 2} > 1 $, a contradiction to our hypothesis.

If $ b \geqslant 2\ell + 1 $, we have that $ |G{:}G'| - |K{:}G'| = b - | K{:}G'| \geqslant \ell + 1 $ by Lemma \ref{numbertheoryresult}, a contradiction to the hypothesis that $ G $ has an irreducible character that vanishes on exactly $ \ell $ conjugacy classes.
\end{proof}
\section*{Acknowledgements}
The author would like to thank the reviewer for providing simpler arguments in the proofs of several lemmas and Theorem A.

\end{document}